\documentclass{amsart}%
\usepackage{amssymb}
\usepackage{amsfonts}
\usepackage{amsmath}
\usepackage{graphicx}%
\newtheorem{theorem}{Theorem}
\theoremstyle{plain}
\newtheorem{acknowledgement}{Acknowledgement}

\newtheorem{lemma}{Lemma}

\newtheorem{proposition}{Proposition}
\newtheorem{remark}{Remark}

\numberwithin{equation}{section}
\begin{document}
\title[Exponential Sums Along $p-$adic Submanifolds]{Exponential Sums and Polynomial Congruences Along $p-$adic Submanifolds}
\author{Dirk Segers}
\address{University of Leuven, Department of Mathematics, Celestijnenlaan 200B, B-3001
Leuven (Heverlee), Belgium.}
\email{dirk.segers@wis.kuleuven.be }
\author{W. A. Z\'{u}\~{n}iga-Galindo}
\address{Centro de Investigaci\'{o}n y de Estudios Avanzados del I.P.N., Departamento
de Matem\'{a}ticas, Av. Instituto Polit\'{e}cnico Nacional 2508, Col. San
Pedro Zacatenco, M\'{e}xico D.F., C.P. 07360, M\'{e}xico. }
\email{wazuniga@math.cinvestav.edu.mx}
\thanks{The first author is a Postdoctoral Fellow of the Fund for Scientific Research
- Flanders (Belgium). The second author was partially supported by CONACYT
(Mexico), Grant \# 127794.}

\begin{abstract}
In this article, we consider the estimation of exponential sums along the
points of the reduction mod $p^{m}$ of a $p$-adic analytic submanifold of
$\mathbb{Z}_{p}^{n}$. More precisely, we extend Igusa's stationary phase
method to this type of exponential sums. We also study the number of solutions
of a polynomial congruence along the points of the reduction mod $p^{m}$ of a
$p$-adic analytic submanifold of $\mathbb{Z}_{p}^{n}$. In addition, we attach
\ a Poincar\'{e} series \ to these numbers, and establish its rationality. In
this way, we obtain geometric bounds for the number of solutions of the
corresponding polynomial congruences.

\end{abstract}
\maketitle

\section{Introduction}

Let $K$ be a $p$-adic field, i.e. $[K:\mathbb{Q}_{p}]<\infty$. Let $R_{K}$ be
the valuation ring of $K$, $P_{K}$ the maximal ideal of $R_{K}$, and
$\overline{K}=R_{K}/P_{K}$ the residue field of $K$. The cardinality of the
residue field of $K$ is denoted by $q$, thus $\overline{K}=\mathbb{F}_{q}$.
For $z\in K$, $ord(z)\in\mathbb{Z}\cup\{+\infty\}$ denotes the valuation,
$|z|_{K}=q^{-ord(z)}$ the $p$-adic norm, and $ac\,z=z\pi^{-\mathrm{ord}(z)}$
the angular component of $z$, where $\pi$ is a fixed uniformizing parameter of
$R_{K}$.

Let $f_{i}\in K[[x_{1},\ldots,x_{n}]]$ be a formal power series for
$i=1,\ldots,l$, with $l\geq2$, and put $x=(x_{1},\ldots,x_{n})$. Let $U$ be an
open and compact subset of $K^{n}$. Assume that each series $f_{i}$ converges
on $U$. We set
\[
V^{(l-1)}:=V^{(l-1)}\left(  K\right)  :=\{x\in U\mid f_{1}(x)=\cdots
=f_{l-1}(x)=0\}
\]
and assume that $V^{(l-1)}$ is a non-empty closed submanifold of $U$, with
dimension $m:=n-l+1\geq1$, which implies that $n\geq l$. We assume that
$f_{l}$ is not identically zero on $V^{(l-1)}$ and that $f_{l}$ has a zero on
$V^{(l-1)}$.We consider on $V^{(l-1)}$ an analytic differential form $\Theta$
of degree $m$, and denote the measure induced on $V^{(l-1)}$ as $\left\vert
\Theta\right\vert $. Later on, we specialize $\Theta$ to a Gel'fand-Leray form
$\gamma_{GL}$ on $V^{(l-1)}$. Let $\Phi:K^{n}\rightarrow\mathbb{C}$ \ be a
Bruhat-Schwartz function with support in $U$. Let $\omega\in\Omega_{0}\left(
K^{\times}\right)  $ be a quasicharacter of $K^{\times}$, see Section
\ref{setquasichar}. To these data we associate the following local zeta
function:
\begin{align*}
Z_{\Phi}(\omega,V^{(l-1)},f_{l})  &  :=Z_{\Phi}(\omega,f_{1},\ldots
,f_{l},\Theta)\\
&  :=\int\limits_{V^{(l-1)}(K)\smallsetminus f_{l}^{-1}(0)}\Phi\left(
x\right)  \omega\left(  f_{l}(x)\right)  \left\vert \Theta\right\vert .
\end{align*}
This function is holomorphic on $\Omega_{0}\left(  K^{\times}\right)  $, and
has a meromorphic continuation to the whole $\Omega\left(  K^{\times}\right)
$ as a rational function of $t=\omega\left(  \pi\right)  =q^{-s}$. The real
parts of the poles of the meromorphic continuation are negative rational
numbers. These assertions follow directly from Igusa's results from the case
$V^{(l-1)}\left(  K\right)  =K^{n-l+1}\cap U$, \cite[Chap. 8]{I2},
alternatively, see \cite[Proposition 2.2]{Z3}. In section \ref{sect2}, by
using a suitable version of Hironaka's resolution theorem (see Theorem
\ref{thresolsing}), we give a list of candidate poles of the local zeta
function in terms of a certain embedded resolution (see Theorem \ref{Th1}).
The purpose of this paper is sharpening some results of \cite{Z3}, mainly
those connecting the poles of local zeta functions of type $Z_{\Phi}%
(\omega,V^{(l-1)},f_{l},\Theta)$ with the estimation of exponential sums and
with the number of solutions of polynomial congruences along $p$-adic
submanifolds, see Theorems \ref{Th2}, \ref{Th3}, \ref{Th4}. We also answer a
question posed in \cite{Z3} about $Z_{\Phi}(\omega,V^{(l-1)},f_{l},\gamma
_{GL})$ and the Dirac delta function (see Theorem \ref{pro1}).

The zeta functions $Z_{\Phi}(\omega,V^{(l-1)},f_{l},\gamma_{GL})$ were studied
in \cite{Z3}, non-Archimedean case, and in \cite{Has}, Archimedean case with
$l=2$, when $f_{1},\ldots,f_{l}$ are non-degenerate with respect to their
Newton polyhedra, see also \cite{G-S}. But if $V^{(l-1)}$ is replaced by
$K^{n}$, the corresponding zeta functions have been extensively studied by
Weil, Tate, Igusa, Denef, Loeser, among others, see e.g. \cite{D0}, \cite{D1},
\cite{I1}, \cite{I2}.

In \cite[Theorem 4.8]{Z3}, Igusa's method for estimating exponential sums was
extended to exponential sums of type:
\[
E(z):=q^{-m(n-l+1)}%
%TCIMACRO{\dsum \limits_{\overline{x}\in V^{(l-1)}\left(  R_{K}\right)
%\operatorname{mod}P_{K}^{m}}}%
%BeginExpansion
{\displaystyle\sum\limits_{\overline{x}\in V^{(l-1)}\left(  R_{K}\right)
\operatorname{mod}P_{K}^{m}}}
%EndExpansion
\Psi\left(  zf_{l}(x)\right)  ,
\]
where $\left\vert z\right\vert _{K}=q^{m\text{ }}$ with $m\in\mathbb{N}$,
$\Psi\left(  \cdot\right)  $ is an additive character of $K$, and $V^{(l-1)}$
is a $p$-adic submanifold of $R_{K}^{n}$ with `good reduction mod $P_{K}$' and
`$f_{1},\ldots,f_{l}$ are non-degenerate with respect to their Newton
polyhedra.' Our main result, and also the main motivation for this paper, is
an extension of Igusa's stationary phase method to exponential sums of type
$E(z)$ without the two above-mentioned conditions,\ see Theorem \ref{Th3}. At
this point, it is worth to mention that the exponential sums along varieties
over finite fields have been extensively studied \cite{A-S}, \cite{Bom},
\cite{Del}, \cite{D-L}, \cite{Gr}, \cite{Per1}, \cite{Per2}, among others. For
exponential sums $\operatorname{mod}p^{m}$, we can mention the references
\cite{Bl}, \cite{D0}, \cite{I1}, \cite{I2}, \cite{Ka}, \cite{Li}, \cite{Mor},
\cite{Z1}, \cite{Z3}, among others. The problem of extending Igusa's
stationary phase method to exponential sums along `varieties
$\operatorname{mod}p^{m} $' was posed by Moreno in \cite{Mor}.

A more \ general problem is to estimate oscillatory integrals of type%

\[
E_{\Phi}(z,V^{(l-1)},f_{l},\Theta):=\int\limits_{V^{(l-1)}(K)}\Phi\left(
x\right)  \Psi\left(  zf_{l}(x)\right)  \mid\Theta\mid,
\]
for $\left\vert z\right\vert _{K}\gg0$. The relevance of studying integrals of
type $E_{\Phi}(z,V^{(l-1)},f_{l},\Theta)$ was pointed out in \cite{Kazh} by
Kazhdan. In this paper we extend Igusa's method to oscillatory integrals of
type $E_{\Phi}(z,V^{(l-1)},f_{l},\Theta)$, more precisely, we show the
existence of an asymptotic expansion for $E_{\Phi}(z,V^{(l-1)},f_{l},\Theta)$,
$\left\vert z\right\vert _{K}\gg0$, which is controlled by the poles of
$Z_{\Phi}(\omega,V^{(l-1)},f_{l},\Theta)$ (see Theorem \ref{Th2}).

We also consider the Poincar\'{e} series associated to the number of solutions
of polynomial congruences along a $p$-adic submanifold of $R_{K}^{n}$ (see
Section \ref{secPoincare}). We show the rationality of a such Poincar\'{e}
series and obtain a bound for the number of solutions of these polynomial
congruences, see Theorem \ref{Th4}, Remark \ref{conjecture}, and \cite{D0},
\cite{De}, \cite{I2}.

\begin{acknowledgement}
The authors thank the anonymous referee for his/her careful reading of the
original version of this paper.
\end{acknowledgement}

\section{Preliminaries}

\subsection{\label{setquasichar}Quasicharacters and local zeta functions}

A quasicharacter of $K^{\times}$ is a continuous group homomorphism from
$K^{\times}$ into $\mathbb{C}^{\times}$. The set of quasicharacters forms an
Abelian group denoted as $\Omega\left(  K^{\times}\right)  $. We set
$\omega_{s}\left(  z\right)  :=\left\vert z\right\vert _{K}^{s}$ for
$s\in\mathbb{C}$ and $z\in K^{\times}$, thus $\omega_{s}\in\Omega\left(
K^{\times}\right)  $. Let now $\omega\in\Omega\left(  K^{\times}\right)  $. If
we choose $s\in\mathbb{C}$\ satisfying $\omega\left(  \pi\right)  =q^{-s}$,
then $\omega\left(  z\right)  =\omega_{s}\left(  z\right)  \chi\left(
ac\text{ }z\right)  $ in which $\chi:=\omega\mid_{R_{K}^{\times}}$ is a
character of $R_{K}^{\times}$, i.e. a continuous group homomorphism from
$R_{K}^{\times}$ into the unit circle of the complex plane. Hence
$\Omega\left(  K^{\times}\right)  $ is a one dimensional complex manifold
since $\Omega\left(  K^{\times}\right)  \cong\mathbb{C\times}\left(
R_{K}^{\times}\right)  ^{\ast}$, where $\left(  R_{K}^{\times}\right)  ^{\ast
}$ is the group of characters of $R_{K}^{\times}$. We note that $\sigma\left(
\omega\right)  :=\operatorname{Re}(s)$ depends only on $\omega$, and
$\left\vert \omega\left(  z\right)  \right\vert =\omega_{\sigma\left(
\omega\right)  }\left(  z\right)  $. We define for every $\sigma\in\mathbb{R}$
an open subset of $\Omega\left(  K^{\times}\right)  $ by
\[
\Omega_{\sigma}\left(  K^{\times}\right)  :=\left\{  \omega\in\Omega\left(
K^{\times}\right)  \mid\sigma\left(  \omega\right)  >\sigma\right\}  .
\]
For further details we refer the reader to \cite{I2}.

Let $f_{i}\in K[[x_{1},\ldots,x_{n}]]$, $i=1,\ldots,l$, with $l\geq2$, $U$,
$V^{(l-1)}$, and $\Phi$ be as in the introduction. We consider on $V^{(l-1)}$,
a closed submanifold of dimension $m$, an analytic differential form $\Theta$
of degree $m$, and denote the measure induced on $V^{(l-1)}$ as $\left\vert
\Theta\right\vert $. We refer the reader to \cite{Ser} and \cite{O} for
further details on $p$-adic manifolds and analytic subsets. Later on, we
specialize $\Theta$ to a Gel'fand-Leray form $\gamma_{GL}$ on $V^{(l-1)}$,
i.e. a form satisfying $\gamma_{GL}\wedge\wedge_{i=1}^{l-1}df_{i}=\wedge
_{i=1}^{n}dx_{i}$. The Gel'fand-Leray form is not unique, but its restriction
to $V^{(l-1)}$ is independent of the choice of $\gamma_{GL}$ (see \cite[Chap.
III, Sect. 1-9]{G-S}). \ By passing to a sufficiently fine covering of the
support of $\Phi$, $Z_{\Phi}(\omega,V^{(l-1)},f_{l})$ can be expressed as a
finite sum of classical Igusa's zeta functions, in this way one verifies that
$Z_{\Phi}(\omega,V^{(l-1)},f_{l})$ is holomorphic on $\omega\in\Omega
_{0}\left(  K^{\times}\right)  $. Since \ any $\omega\in\Omega\left(
K^{\times}\right)  $ can be expressed as $\omega\left(  z\right)  =\chi\left(
ac\text{ }z\right)  \left\vert z\right\vert _{K}^{s}$, we use the notation
$Z_{\Phi}(s,\chi):=Z_{\Phi}(s,\chi,V^{(l-1)},f_{l}):=Z_{\Phi}(\omega
,V^{(l-1)},f_{l})$.

\subsection{Resolution of singularities}

The following version of Hironaka's resolution theorem will be used later on:

\begin{theorem}
[Hironaka, \cite{H}]\label{thresolsing}There exists an embedded
resolution\textbf{\ } $\sigma:Y\rightarrow V^{(l-1)}$ of $f_{l}:V^{(l-1)}%
\rightarrow K$ and $\Theta$, that is,

\noindent(1) $Y$ is an $m-$dimensional $K-$analytic compact manifold, and
$\sigma$ is a proper $K-$analytic map which is an isomorphism outside of
$S:=f_{l}^{-1}\left(  0\right)  $;

\noindent(2) $\sigma^{-1}\left(  S\right)  =\cup_{i\in T}E_{i}$, where \ the
$E_{i}$\ are closed submanifolds of $Y$ of codimension one, each equipped with
a pair of positive integers $\left(  N_{i},v_{i}\right)  $ satisfying the
following: at every \ point $b$ of $Y$ there exist local coordinates $\left(
y_{1},\ldots,y_{m}\right)  $ on $Y$ around $b$ such that, if $E_{1}%
,\ldots,E_{k}$ are the $E_{i}$ containing $b$, we have \ on some neighborhood
of $b$ that $E_{i}$ is given by $y_{i}=0$ for $i=1,\ldots,k$,%
\[
f_{l}\circ\sigma=\text{ }\varepsilon\left(  y\right)
%TCIMACRO{\dprod \limits_{i=1}^{k}}%
%BeginExpansion
{\displaystyle\prod\limits_{i=1}^{k}}
%EndExpansion
y_{i}^{N_{i}},
\]
and
\[
\sigma^{\ast}\Theta=\eta\left(  y\right)  \left(
%TCIMACRO{\dprod \limits_{i=1}^{k}}%
%BeginExpansion
{\displaystyle\prod\limits_{i=1}^{k}}
%EndExpansion
y_{i}^{v_{i}-1}\right)  dy_{1}\wedge\ldots\wedge dy_{m},
\]
where \ $\varepsilon\left(  y\right)  $, $\eta\left(  y\right)  $ are units in
the local ring of $Y$\ at $b$.
\end{theorem}

The above theorem is a variation of Theorem 2.2 in \cite{D-V}, and its proof
follows from Corollary 3 in \cite{H} by using the reasoning given at the
bottom of p. 97 in \cite{D-V}.

We call the $\left(  N_{i},v_{i}\right)  $, $i\in T$, \textit{the numerical
data of} $\left(  \sigma,\Theta\right)  $. From now on, we fix $\Theta$ and
say that $\left(  N_{i},v_{i}\right)  ,i\in T$, are the numerical data of
$\sigma:Y\rightarrow V^{(l-1)}$. Set $\rho:=\rho(V^{(l-1)},f_{l},\Theta
):=\min_{i\in T}{v_{i}/N_{i}}$.

We denote the set of critical points of the map $f_{l}:V^{(l-1)}\rightarrow K$
by $C_{f_{l}}$, i.e.
\begin{align*}
C_{f_{l}}  &  =\left\{  x\in V^{(l-1)}\mid rank_{K}\left(  \frac{\partial
f_{i}}{\partial x_{j}}(x)\right)  _{\substack{1\leq i\leq l\\1\leq j\leq
n}}\leq l-1\right\} \\
&  =\left\{  x\in V^{(l-1)}\mid rank_{K}\left(  \frac{\partial f_{i}}{\partial
x_{j}}(x)\right)  _{\substack{1\leq i\leq l\\1\leq j\leq n}}=l-1\right\}  ,
\end{align*}
since $rank_{K}\left(  \frac{\partial f_{i}}{\partial x_{j}}(x)\right)
_{1\leq i\leq l-1,\text{ }1\leq j\leq n}=l-1$ for $x\in V^{(l-1)}$. For
$\lambda\in K^{\times}$, we set $V^{(l,\lambda)}:=\{x\in U\mid f_{1}%
(x)=\ldots=f_{l-1}(x)=0,\,f_{l}(x)=\lambda\}$.

\begin{remark}
Note that the following three statements are equivalent: (1) $C_{f_{l}}%
\subset{f_{l}}^{-1}(0)$; (2) for every $\lambda\in K^{\times}$, the $l$-form
$\bigwedge_{i=1}^{l}df_{i}$ does not vanish on $V^{(l,\lambda)} $, and (3) for
every $\lambda\in K^{\times}$, $V^{(l,\lambda)}$ is a closed submanifold of
dimension $n-l$.
\end{remark}

\section{\label{sect2}Poles of Local Zeta Functions Supported on $p$-adic
Submanifolds}

\begin{theorem}
\label{Th1}Let $\sigma:Y\rightarrow V^{(l-1)}$ be a fixed embedded resolution
of $f_{l}:V^{(l-1)}\rightarrow K$, with \textit{numerical data} $\left(
N_{i},v_{i}\right)  $, $i\in T$. Then

\noindent(1) $Z_{\Phi}(s,\chi,V^{(l-1)},f_{l})$ has a meromorphic continuation
as a rational function of $q^{-s}$. Its poles are among the values
\[
s=-\frac{v_{i}}{N_{i}}-\frac{2\pi\sqrt{-1}}{N_{i}\log q}k\text{, with }%
k\in\mathbb{Z}\text{,and }i\in T\text{,}%
\]
such that the order of $\chi$ divides $N_{i}$;

\noindent(2) if $C_{f_{l}}\subset f_{l}^{-1}(0)$, then there exists
$e(\Phi)>0$ in $\mathbb{N}$ such that
\[
Z_{\Phi}(s,\chi,V^{(l-1)},f_{l})=0\text{ \ \ for every \ }s\in\mathbb{C}%
\text{,}%
\]
unless the conductor $c(\chi)$ of $\chi$ satisfies $c(\chi)\leq e(\Phi)$;

\noindent(3) $-\rho$ is the real part of a pole of $Z_{\Phi}(s,\chi
_{triv},V^{(l-1)},f_{l})$, for some $\Phi$ with support in $U$, and
consequently, $\rho$ is independent of the embedded resolution chosen.
\end{theorem}

\begin{proof}
\noindent(1) The proof uses the same argument of the case $V^{(l-1)}\left(
K\right)  =K^{n-l+1}\cap U$, see \cite[Theorem 8.2.1]{I2}.\ (2) The proof is a
variation of the one given for Theorem 4.4 in \cite{Z3}.

\noindent(3) The proof is analogous to the one given for Theorem 2.7 in
\cite{VZ}.
\end{proof}

\begin{remark}
\label{note}We set $m_{\rho}:=m_{\rho}(V^{(l-1)},f_{l},\Theta)$ for the
largest multiplicity of the poles of $(1-q^{-1-s})Z_{\Phi}(s,\chi
_{\mathrm{triv}},V^{(l-1)},f_{l},\Theta)$ and $Z_{\Phi}(s,\chi,V^{(l-1)}%
,f_{l},\Theta)$ with $\chi\not =\chi_{triv}$ having real part $-\rho$, when
$\Phi$ runs through all the Bruhat-Schwartz functions. Note that by the
previous theorem $\rho$ and $m_{\rho}$ are well-defined.
\end{remark}

\section{The oscillatory integrals $E_{\Phi}(z)$}

\subsection{Additive characters}

Given $z=\sum_{n=n_{0}}^{\infty}z_{n}p^{n}\in\mathbb{Q}_{p}$, with $z_{n}%
\in\left\{  0,\ldots,p-1\right\}  $ and $z_{n_{0}}\neq0$, we set
\[
\left\{  z\right\}  _{p}:=\left\{
\begin{array}
[c]{lll}%
0 & \text{if} & n_{0}\geq0\\
&  & \\
\sum_{n=n_{0}}^{-1}z_{n}p^{n} & \text{if} & n_{0}<0,
\end{array}
\right.
\]
\textit{the fractional part of }$z$. Then $\exp(2\pi\sqrt{-1}\left\{
z\right\}  _{p}),$ $z\in\mathbb{Q}_{p}$, is an additive character on
$\mathbb{Q}_{p}$ trivial on $\mathbb{Z}_{p}$ but not on $p^{-1}\mathbb{Z}_{p}$.

We recall that there exists an integer $d\geq0$ such that $Tr_{K/\mathbb{Q}%
_{p}}(z)\in\mathbb{Z}_{p}$ for $\left\vert z\right\vert _{K}\leq q^{d}$ but
$Tr_{K/\mathbb{Q}_{p}}(z_{0})\notin\mathbb{Z}_{p}$ for some $z_{0}$\ with
$\left\vert z_{0}\right\vert _{K}=q^{d+1}$. The integer $d$ is called
\textit{the exponent of the different} of $K/\mathbb{Q}_{p}$. It is known that
$d\geq e-1$, where $e$ is the ramification index of $K/\mathbb{Q}_{p}$, see
e.g. \cite[Chap. VIII, Corollary of Proposition 1]{W}. The additive character%
\[
\varkappa(z)=\exp(2\pi\sqrt{-1}\left\{  Tr_{K/\mathbb{Q}_{p}}(\pi
^{-d}z)\right\}  _{p}),\text{ }z\in K\text{, }%
\]
is \textit{a standard character} of $K$, i.e. $\varkappa$ is trivial on
$R_{K}$ but not on $P_{K}^{-1}$. For our purposes, it is more convenient to
use
\[
\Psi(z)=\exp(2\pi\sqrt{-1}\left\{  Tr_{K/\mathbb{Q}_{p}}(z)\right\}
_{p}),\text{ }z\in K\text{, }%
\]
instead of $\varkappa(\cdot)$. This particular choice is due to the fact that
we use Denef's approach for estimating oscillatory integrals, see
\cite[Proposition 1.4.4]{D0}.

\subsection{\label{sect4}Asymptotic expansion of oscillatory integrals}

We set
\[
E_{\Phi}(z):=E_{\Phi}(z,V^{(l-1)},f_{l},\Theta)=\int\limits_{V^{(l-1)}\left(
K\right)  }\Phi(x)\Psi(zf_{l}(x))|\Theta|,
\]
for $z\in K$. The following theorem describes the asymptotic behavior of
oscillatory integrals $E_{\Phi}(z)$.

\begin{theorem}
\label{Th2}Let $\sigma:Y\rightarrow V^{(l-1)}$ be any embedded resolution of
$f_{l}:V^{(l-1)}\rightarrow K$, with \textit{numerical data} $\left(
N_{i},v_{i}\right)  $, $i\in T$. Assume that $C_{f_{l}} \subset f_{l}^{-1}%
(0)$. Then

\noindent(1) for $\left\vert z\right\vert _{K}$ big enough $E_{\Phi}(z)$ is a
finite $\mathbb{C}$-linear combination \ of functions of the form $\chi\left(
ac\text{ }z\right)  \left\vert z\right\vert _{K}^{\lambda}\left(  \log
_{q}\left\vert z\right\vert _{K}\right)  ^{\beta}$ with coefficients
independent of $z$, and with $\lambda\in\mathbb{C}$ a pole of $\left(
1-q^{-1-s}\right)  Z_{\Phi}(s,\chi_{\mathrm{triv}},V^{(l-1)},f_{l})$ or of
$Z_{\Phi}(s,\chi,V^{(l-1)},f_{l})$, and \ $\beta\in\mathbb{N}$, with
$\beta\leq\left(  \text{multiplicity of }\lambda\right)  $ $-1$. In addition,
all poles $\lambda$ appear effectively in this linear combination.

\noindent(2) There exists a constant $C$ such that for $|z|_{K}>1$,
\[
\left\vert E_{\Phi}(z)\right\vert \leq C\left\vert z\right\vert _{K}^{-\rho
}\left(  \log_{q}\left\vert z\right\vert _{K}\right)  ^{m_{\rho}-1}.
\]

\end{theorem}

\begin{proof}
(1) Let Coeff$_{t^{k}}Z_{\Phi}(s,\chi)$ denote the coefficient of $t^{k}$ in
the power expansion of $Z_{\Phi}(s,\chi)$ in the variable $t=q^{-s}$. The
following formula is a variation of Proposition 1.4.4 given by Denef in
\cite{D0}, see also \cite[Proposition 4.6]{Z2}: \ for $u\in R^{\times}$ and
$m\in\mathbb{Z}$,
\begin{align}
E_{\Phi}(u\pi^{-m})  &  =Z_{\Phi}(0,\chi_{\mathrm{triv}})+\text{Coeff}%
_{t^{m-1}}\frac{(t-q)Z_{\Phi}(s,\chi_{\mathrm{triv}})}{(q-1)(1-t)}%
\label{Form1}\\
&  +\sum\limits_{\chi\neq\chi_{\mathrm{triv}}}g_{\chi^{-1}}\chi(u)\text{Coeff}%
_{t^{m-c(\chi)}}Z_{\Phi}(s,\chi),\nonumber
\end{align}
where $c(\chi)$ is the conductor of $\chi$, and $g_{\chi}$ denotes the
Gaussian sum
\[
g_{\chi}=(q-1)^{-1}q^{1-c(\chi)}\sum\limits_{v\in(R/P^{c(\chi)})^{\times}}%
\chi(v)\Psi(v/\pi^{c(\chi)}).
\]
By using the hypothesis $C_{f_{l}}\subset f_{l}^{-1}(0)$, $Z_{\Phi}(s,\chi)$
is a rational function identically zero for almost all $\chi$ (cf. Theorem
\ref{Th1} ), hence the series in the right side of (\ref{Form1}) is a finite
sum. The asymptotic expansion for $E_{\Phi}(z)$ is obtained by expanding the
right side of (\ref{Form1}) in partial fractions.

(2) The estimation for $|z|_{K}$ big enough is obtained as follows: there
exist $\rho$, $m_{\rho}$ (cf. Remark \ref{note}) such that for every pole
$\lambda$ in Theorem \ref{Th2} (1),%
\[
\left\vert z\right\vert _{K}^{\lambda}\left(  \log_{q}\left\vert z\right\vert
_{K}\right)  ^{\beta}\leq C\left\vert z\right\vert _{K}^{-\rho}\left(
\log_{q}\left\vert z\right\vert _{K}\right)  ^{m_{\rho}-1},
\]
for $|z|_{K}$ big enough, \ and some constant $C$. The estimation of
$|E_{\Phi}(z)|$, for $|z|_{K}>1$, follows from the previous estimation by
adjusting the constant $C$, since $|E_{\Phi}(z)|$ is upper bounded.
\end{proof}

The estimation given in the second part of Theorem \ref{Th2} is optimal, in
the sense that there exists a $\Phi$ such that the constants $\rho$, $m_{\rho
}$ cannot be improved.

\section{Exponential Sums Along $p$-adic Submanifolds of $R_{K}^{n}$}

\subsection{Some additional notation}

From now on, we assume that all the $f_{i}\left(  x\right)  $ have
coefficients in $R_{K}$, and put $U=R_{K}^{n}$, and set
\[
V^{(j)}\left(  R_{K}\right)  :=\left\{  x\in R_{K}^{n}\mid f_{i}\left(
x\right)  =0,\text{ }i=1,\ldots,j\right\}
\]
for $j=l-1$, $l$. Note that $V^{(l-1)}\left(  R_{K}\right)  $ is a closed
submanifold of dimension $n-l+1\geq1$.

Let $\operatorname{mod}P_{K}^{m}$ denote the canonical homomorphism $R_{K}%
^{n}\rightarrow\left(  R_{K}/P_{K}^{m}\right)  ^{n}$, for $m\in
\mathbb{N\smallsetminus}\left\{  0\right\}  $. The image of $x\in R_{K}^{n}$
under this homomorphism is denoted by $\overline{x}$. We will call the image
of $A\subseteq R_{K}^{n}$ by $\operatorname{mod}P_{K}^{m}$, \textit{the
reduction }$\operatorname{mod}P_{K}^{m}$\textit{\ of }$A$, and it will be
denoted as $A$ $\operatorname{mod}P_{K}^{m}$. When we write $\overline{x}\in
A\operatorname{mod}P_{K}^{m}$, we always assume without mentioning that $x\in
A$. We will apply these definitions for $A$ equal to $V^{(l-1)}(R_{K})$.

For any polynomial $g$ over $R_{K}$ we denote by $\overline{g}$ \ the
polynomial over $\overline{K}$ obtained by reducing each coefficient of $g$
modulo $P_{K}$.

We define for $m\in\mathbb{N\smallsetminus}\left\{  0\right\}  $ the set
\[
V^{(l-1)}\left(  R_{K}/P_{K}^{m}\right)  =\left\{  \overline{x}\in\left(
R_{K}/P_{K}^{m}\right)  ^{n}\mid ord\left(  f_{i}\left(  x\right)  \right)
\geq m,\text{ }i=1,\ldots,l-1\right\}  .
\]
We note that \textquotedblleft$ord\left(  f_{i}\left(  x\right)  \right)  \geq
m$\textquotedblright\ is independent of the representative chosen to compute
$ord\left(  f_{i}\left(  x\right)  \right)  $. For $m=1$, we write often
$V^{(l-1)}(\overline{K})$ instead of $V^{(l-1)}(R_{K}/P_{K})$. Note that
\[
V^{(l-1)}\left(  \overline{K}\right)  =\left\{  x\in\overline{K}^{n}%
\mid\overline{f_{i}}\left(  x\right)  =0,\text{ }i=1,\ldots,l-1\right\}  .
\]

We will say that $V^{(l-1)}\left(  R_{K}\right)  $ \textit{has good reduction}
$\operatorname{mod}$ $P_{K}$ if
\[
rank_{\overline{K}}\left[  \overline{\frac{\partial f_{i}}{\partial x_{j}}%
}\left(  x\right)  \right]  _{\substack{1\leq i\leq l-1 \\1\leq j\leq
n}}=l-1\text{, for every }x\in V^{(l-1)}\left(  \overline{K}\right)  .
\]
If $V^{(l-1)}\left(  R_{K}\right)  $ has good reduction $\operatorname{mod}$
$P_{K}$, the Hensel lemma implies that
\[
V^{(l-1)}\left(  R_{K}\right)  \operatorname{mod}P_{K}^{m}=V^{(l-1)}\left(
R_{K}/P_{K}^{m}\right)
\]
for every $m\in\mathbb{N}\setminus\{0\}$.

\subsection{Submanifolds with bad reduction $\operatorname{mod}P_{K}$}

In general the \ reduction $V^{(l-1)}\left(  R_{K}\right)  \operatorname{mod}%
P_{K}$ has singular points. In order to deal with these points we use some
ideas about N\'{e}ron $\pi$-desingularization, see e.g. \cite[Section 4]{A},
or \cite[Proposition 2.4 and Lemma 2.5]{Z2}. Since we did not find a suitable
reference for our purposes, we prove below the required results.

\begin{lemma}
Suppose that $V^{\left(  l-1\right)  }\left(  R_{K}\right)  $ is a closed
submanifold of dimension $n-l+1$. Let $x_{0}\in V^{\left(  l-1\right)
}\left(  R_{K}\right)  $. Then, there exist a positive integer $L$ and
polynomials $f_{1,x_{0}},\ldots,f_{l-1,x_{0}}\in R_{K}\left[  x_{1}%
,\ldots,x_{n}\right]  $ such that%
\[
V^{(l-1)}\left(  R_{K}\right)  =\left\{  x\in R_{K}^{n}\mid f_{i,x_{0}}\left(
x\right)  =0,\text{ }i=1,\ldots,l-1\right\}  ,
\]
and if we define
\[
f_{i,x_{0}}\left(  x_{0}+\pi^{L}y\right)  =\pi^{e_{i,x_{0},L}}f_{i,x_{0}%
,L}\left(  y\right)  \text{, }%
\]
with $f_{i,x_{0},L}\left(  y\right)  \in R_{K}\left[  y_{1},\ldots
,y_{n}\right]  \setminus P_{K}\left[  y_{1},\ldots,y_{n}\right]  $, for
$i=1,\ldots,l-1$, then, the submanifold
\[
V_{x_{0},L}^{(l-1)}\left(  R_{K}\right)  :=\left\{  y\in R_{K}^{n}\mid
f_{i,x_{0},L}\left(  y\right)  =0,\text{ }i=1,\ldots,l-1\right\}
\]
has good reduction $\operatorname{mod}P_{K}$.
\end{lemma}

\begin{proof}
By applying a translation, we can assume that $x_{0}$ is the origin. We set%
\[
f_{i}\left(  x\right)  =a_{i,1}x_{1}+\ldots+a_{i,n}x_{n}+\text{ higher degree
terms,}%
\]
for $i=1,\ldots,l-1$. The matrix $\left(  a_{i,j}\right)  $ has rank $l-1$
over $K$ because $V^{\left(  l-1\right)  }\left(  R_{K}\right)  $ is a closed
submanifold of dimension $n-l+1$. The announced polynomials $f_{i,x_{0}}$ are
linear combinations with coefficients in $R_{K}$ of the $f_{i}$. These linear
combinations are determined by the elementary row operations over $R_{K}$
required to reduced the matrix $\left(  a_{i,j}\right)  $ to its row echelon
form. Since $R_{K}$ is not a field some details are required. We select an
entry $a_{i_{0},j_{0}}$ of $\left(  a_{i,j}\right)  $ satisfying $ord\left(
a_{i_{0},j_{0}}\right)  =\min_{i,j}ord\left(  a_{i,j}\right)  $. Then by row
and column interchanging one gets a new matrix having $a_{i_{0},j_{0}} $ in
the position $\left(  1,1\right)  $. Thus we can assume that $i_{0}=j_{0}=1$.
In addition, by row interchanging we can assume that%
\[
ord\left(  a_{2,1}\right)  \leq ord\left(  a_{3,1}\right)  \leq\ldots\leq
ord\left(  a_{l-1,1}\right)  .
\]
We now can perform elementary row operations on $R_{K}$ to obtain a matrix
$\left(  a_{i,j}^{\prime}\right)  $ satisfying $a_{1,1}^{\prime}=a_{1,1}$,
$a_{2,1}^{\prime}=\ldots=a_{l-1,1}^{\prime}=0$. We now apply the previous
procedure to $\left(  a_{i,j}^{\prime}\right)  $, $2\leq i\leq l-1$, $1\leq
j\leq n$. By using this procedure, we construct a matrix $\left(
b_{i,j}\right)  $ which is the row echelon form\ with $rank_{K}$ $\left(
b_{i,j}\right)  =rank_{K}\left(  a_{i,j}\right)  =l-1$ and define%
\[
f_{i,x_{0}}\left(  x\right)  =b_{i,i}x_{i}+\ldots+b_{i,n}x_{n}+\text{ higher
degree terms,}%
\]
for $i=1,\ldots,l-1$. Note that
\[
ord\left(  b_{1,1}\right)  \leq ord\left(  b_{2,2}\right)  \leq\ldots\leq
ord\left(  b_{l-1,l-1}\right)
\]
with $b_{i,i}\neq0$, 1$\leq i\leq l-1$, because $rank_{K}$ $\left(
b_{i,j}\right)  =rank_{K}\left(  a_{i,j}\right)  =l-1$. In addition,
\[
ord(b_{i,i})\leq ord\left(  b_{i,j}\right)  \text{, for }i<j\leq n.
\]
We set
\[
L:=ord\left(  b_{l-1,l-1}\right)  +1<\infty,
\]
and%
\[
f_{i,x_{0}}(\pi^{L}y):=\pi^{e_{i,x_{0},L}}f_{i,x_{0},L}\left(  y\right)  ,
\]
where $e_{i,x_{0},L}=L+ord\left(  b_{i,i}\right)  $, $f_{i,x_{0},L}\left(
y\right)  =c_{i,i}x_{i}+c_{i,i+1}x_{i+1}+\ldots+c_{i,n}x_{n}+$ $($higher
degree terms$)$, $c_{i,i}=ac\left(  b_{i,i}\right)  \in R_{K}^{\times}$, for
$i=1,\ldots,l-1$. Note that $rank_{K}$ $\left(  b_{i,j}\right)  =rank_{K}%
\left(  c_{i,j}\right)  =l-1$, and that $rank_{\overline{K}}\left(
\overline{c}_{i,j}\right)  =l-1$.

Finally, since%
\[
rank_{\overline{K}}\left(  \frac{\partial\overline{f_{i,x_{0},L}}}{\partial
y_{j}}(\overline{x})\right)  =rank_{\overline{K}}\left(  \overline{c}%
_{i,j}\right)  =l-1,
\]
for every $\overline{x}\in V_{x_{0},L}^{(l-1)}\left(  \overline{K}\right)  $,
we conclude that $V_{x_{0},L}^{(l-1)}\left(  R_{K}\right)  $ has good
reduction $\operatorname{mod}P_{K}$.
\end{proof}

In the above proof, we can take for each $x^{\prime}=x_{0}+\left(  P_{K}%
^{L+1}\right)  ^{n}\cap V^{\left(  l-1\right)  }\left(  R_{K}\right)  $ the
same $L$. Thus by the compactness of $V^{\left(  l-1\right)  }\left(
R_{K}\right)  $ there are only finitely many $L$ involved, and consequently,
by taking the maximum of these numbers, we can take $L$ independently of
$x_{0}$. In this way we get the following result.

\begin{proposition}
\label{Neron_resol}Assume that $V^{\left(  l-1\right)  }\left(  R_{K}\right)
$ is a closed submanifold of dimension $n-l+1$. Then, there exist a positive
integer $L$ such that for every $x_{0}\in V^{\left(  l-1\right)  }\left(
R_{K}\right)  $, there exist equations $f_{1,x_{0}}\left(  x\right)
=\ldots=f_{l-1,x_{0}}(x)=0$ \ with coefficients in $R_{K}$ defining
$V^{\left(  l-1\right)  }\left(  R_{K}\right)  $ such that if we write
$f_{i,x_{0}}\left(  x_{0}+\pi^{L}y\right)  =\pi^{e_{i,x_{0},L}}f_{i,x_{0}%
,L}\left(  y\right)  $, with $f_{i,x_{0},L}\left(  y\right)  \in R_{K}\left[
y_{1},\ldots,y_{n}\right]  \setminus P_{K}\left[  y_{1},\ldots,y_{n}\right]
$, for $i=1,\ldots,l-1$,the submanifold $V_{x_{0},L}^{(l-1)}\left(
R_{K}\right)  =\left\{  y\in R_{K}^{n}\mid f_{i,x_{0},L}\left(  y\right)
=0,\text{ }i=1,\ldots,l-1\right\}  $ has \ good reduction $\operatorname{mod}%
P_{K}$.
\end{proposition}

\subsection{Bounding $|E(z)|$}

In this section, we use $\rho=\rho(V^{(l-1)},f_{l},\gamma_{GL})$ and $m_{\rho
}=m_{\rho}(V^{(l-1)},f_{l},\gamma_{GL})$ which were defined before, see Remark
\ref{note}.

We set for $z=u\pi^{-m}\in K$, with $u\in R^{\times}$ and $m\in\mathbb{N}%
\setminus\{0\}$, as in the introduction, the exponential sum
\[
E(z)=q^{-m\left(  n-l+1\right)  }\text{{\ \ }}{\sum\limits_{\overline{x}\in
V^{(l-1)}\left(  R_{K}\right)  \operatorname{mod}P_{K}^{m}}}\Psi\left(
zf_{l}(x)\right)  .
\]

\begin{theorem}
\label{Th3} Assume that $C_{f_{l}}\subset f_{l}^{-1}(0)$. Then there exists a
constant $C$ such that
\[
\left\vert E(z)\right\vert \leq C\left\vert z\right\vert _{K}^{-\rho}\left(
\log_{q}\left\vert z\right\vert _{K}\right)  ^{m_{\rho}-1},
\]
for $|z|_{K}>1$.
\end{theorem}

\begin{proof}
Consider first the case in which $V^{(l-1)}(R_{K})$ has good reduction
$\operatorname{mod}P_{K}$. By applying Lemma 5.5 in \cite{Z3},
\[
E(z)=\int\limits_{V^{(l-1)}(R_{K})}\Psi\left(  zf_{l}\left(  x\right)
\right)  \mid\gamma_{GL}\mid,
\]
now, the announced estimation follows from Theorem \ref{Th2}.

We now consider the general case. We use Proposition \ref{Neron_resol} to
reduce the estimation of $\left\vert E(z)\right\vert $ to the estimation of
several exponential sums along several submanifolds with good reduction
$\operatorname{mod}P_{K}$, as follows. There exists a positive integer $L$
such that for every $x\in V^{(l-1)}(R_{K})$, there exist equations
$f_{1,x},\ldots,f_{l-1,x}$ of $V^{(l-1)}(R_{K})$ such that if we write
$f_{i,x}(x+\pi^{L}y)=\pi^{e_{i,x,L}}f_{i,x,L}(y)$ for $i=1,\ldots,l-1$ , with
$f_{i,x,L}(y)\in R_{K}[y_{1},\ldots,y_{n}]\setminus P_{K}[y_{1},\ldots,y_{n}%
]$, the manifold
\[
V_{x,L}^{(l-1)}(R_{K})=\{y\in R_{K}^{n}\mid f_{i,x,L}(y)=0,i=1,\ldots,l-1\}
\]
has good reduction mod $P_{K}$. Take such an integer $L$. For $x\in
V^{(l-1)}(R_{K})$, we write $f_{l}(x+\pi^{L}y)=f_{l}(x)+\pi^{e_{l,x,L}%
}f_{l,x,L}(y)$, with $f_{l,x,L}(y)\in R_{K}[y_{1},\ldots,y_{n}]\setminus
P_{K}[y_{1},\ldots,y_{n}]$. Note that $e_{l,x,L}\geq L$.

For $z=u\pi^{-m}$, with $u\in R_{K}^{\times}$ and $m>L$, we get
\begin{align*}
\lefteqn{q^{m(n-l+1)}E(u\pi^{-m})}\\
&  =\sum_{\overline{x}\in V^{(l-1)}(R_{K})\operatorname{mod}P_{K}^{m}}%
\Psi(uf_{l}(x)/\pi^{m})\\
&  =\sum_{\overline{x}\in V^{(l-1)}(R_{K})\mathrm{\operatorname{mod}}P_{K}%
^{L}}\text{ \ \ }\sum_{\overline{y}\in V_{x,L}^{(l-1)}(R_{K}%
)\operatorname{mod}P_{K}^{m-L}}\Psi(uf_{l}(x+\pi^{L}y)/\pi^{m})\\
&  =\sum_{\overline{x}\in V^{(l-1)}(R_{K})\operatorname{mod}P_{K}^{L}}\text{
}\Psi(\frac{uf_{l}(x)}{\pi^{m}})\text{\ \ }\sum_{\overline{y}\in
V_{x,L}^{(l-1)}(R_{K})\operatorname{mod}P_{K}^{m-L}}\Psi(\frac{u\pi
^{e_{l,x,L}-L}f_{l,x,L}(y)}{\pi^{m-L}}).
\end{align*}

Therefore,%
\begin{align*}
\left\vert E(u\pi^{-m})\right\vert  &  \leq\\
&  C_{0}\max_{\overline{x}\in V^{(l-1)}(R_{K})\operatorname{mod}P_{K}^{L}%
}\left\vert \sum_{\overline{y}\in V_{x,L}^{(l-1)}(R_{K})\operatorname{mod}%
P_{K}^{m-L}}\Psi(\frac{u\pi^{e_{l,x,L}-L}f_{l,x,L}(y)}{\pi^{m-L}})\right\vert
,
\end{align*}
where $C_{0}=\#\left(  V^{(l-1)}(R_{K})\operatorname{mod}P_{K}^{L}\right)  $.

Since $\theta_{x}:K^{n}\rightarrow K^{n}$, with $\theta_{x}(y)=x+\pi^{L}y$ is
a $K-$analytic isomorphism for any $x\in R_{K}^{n}$, and $\theta_{x}\left(
V_{x,L}^{\left(  l-1\right)  }\left(  R_{K}\right)  \right)  =V^{(l-1)}\left(
R_{K}\right)  \cap\left(  x+\left(  P_{K}^{L}\right)  ^{n}\right)  $, we have%
\begin{align}
\rho &  =\rho(V^{(l-1)},f_{l,},\gamma_{GL})\label{obs2}\\
&  \leq\rho(V_{x,L}^{(l-1)},f_{l,x,L},\gamma_{GL})\text{, for any }%
\overline{x}\in V^{(l-1)}(R_{K})\operatorname{mod}P_{K}^{L}\text{,}\nonumber
\end{align}%
\begin{equation}
\rho=\rho(V_{x,L}^{(l-1)},f_{l,x,L},\gamma_{GL})\text{, for some }\overline
{x}\in V^{(l-1)}(R_{K})\operatorname{mod}P_{K}^{L}\text{,} \label{obs2a}%
\end{equation}%
\begin{align}
m_{\rho}  &  =m_{\rho}(V^{(l-1)},f_{l},\gamma_{GL})\label{obs3}\\
&  \geq m_{\rho}(V_{x,L}^{(l-1)},f_{l,x,L},\gamma_{GL})\text{, for any
}\overline{x}\in V^{(l-1)}(R_{K})\operatorname{mod}P_{K}^{L}\text{.}\nonumber
\end{align}
We now note that all the $V_{x,L}^{(l-1)}\left(  R_{K}\right)  $ have good
reduction $\operatorname{mod}P_{K}$, and then by applying the estimation given
at the beginning of the proof, we have%
\begin{equation}
|E(u\pi^{-m})|\leq C_{0}q^{-m\rho}m^{m_{\rho}-1}\text{, for }m\geq L\geq1.
\label{inequality}%
\end{equation}
Finally, since
\begin{align*}
|E(u\pi^{-m})|  &  \leq q^{-m\left(  n-l+1\right)  }\#\left(  V^{(l-1)}%
(R_{K})\operatorname{mod}P_{K}^{m}\right) \\
&  \leq q^{m\left(  l-1\right)  },
\end{align*}
for every $m$, we can replace $C_{0}$ by%
\[
C:=\max\left\{  \left\{  C_{0}\right\}  \cup\left\{  \frac{q^{m\left(
l-1\right)  }}{q^{-m\rho}m^{m_{\rho}-1}}\mid1\leq m\leq L-1\right\}  \right\}
,
\]
in (\ref{inequality}), and thus the estimation holds for all $m\geq1$.
\end{proof}

\section{\label{secPoincare}Poincar\'{e} Series and Polynomial Congruences
Along $p$-adic Submanifolds of $R_{K}^{n}$}

We define for $m\in\mathbb{N}$ the number $N_{m}=N_{m}\left(  V^{(l-1)}%
,f_{l}\right)  $ as
\begin{equation}
\left\{
\begin{array}
[c]{ll}%
\#\left(  \left\{  \overline{x}\in V^{(l-1)}\left(  R_{K}\right)  \text{
}\operatorname{mod}P_{K}^{m}\mid ord\left(  f_{l}\left(  x\right)  \right)
\geq m\right\}  \right)  & \text{ if }m\geq1\\
& \\
1 & \text{ if }m=0.
\end{array}
\right.  \label{Nm}%
\end{equation}
Note that $ord(f_{l}(x))\geq m$, if and only if $f_{l}(x)\equiv
0\operatorname{mod}P_{K}^{m}$, and therefore, the $N_{m}$ give the number of
solutions of a polynomial congruence along the submanifold $V^{(l-1)}(R_{K})$.
We also define
\[
P\left(  t\right)  :=P\left(  t,V^{(l-1)},f_{l}\right)  ={\sum\limits_{m=0}%
^{\infty}}q^{-m\left(  n-l+1\right)  }N_{m}t^{m}.
\]
If $V^{(l-1)}\left(  R_{K}\right)  $ has good reduction $\operatorname{mod}$
$P_{K}$, then
\[
N_{m}=\#(\left\{  \overline{x}\in\left(  R_{K}/P_{K}^{m}\right)  ^{n}\mid
f_{1}\left(  x\right)  \equiv f_{2}\left(  x\right)  \equiv\ldots\equiv
f_{l}\left(  x\right)  \equiv0\operatorname{mod}P_{K}^{m}\right\}  ).
\]

In the following theorem, we prove the rationality of $P(t)$ and give an upper
bound for the $N_{m}$.

\begin{theorem}
\label{Th4}(1) $P\left(  t\right)  $ is a rational function of $q^{-s}$. (2)
There exists a constant $C$ such that
\[
N_{m}\leq Cq^{(n-l+1-\rho)m}m^{m_{\rho}-1},
\]
for all $m\geq1$.
\end{theorem}

\begin{proof}
We first prove (1) and (2) assuming that $V^{(l-1)}(R_{K})$ has good reduction
mod $P_{K}$. By Lemma 5.3 in \cite{Z3},
\begin{equation}
P\left(  t\right)  =\frac{1-tZ_{\Phi}\left(  s,\chi_{\mathrm{triv}}%
,V^{(l-1)},f_{l}\right)  }{1-t}, \label{for2}%
\end{equation}
with $t=q^{-s}$ and $\Phi$ the characteristic function of $R_{K}^{n}$. The
rationality follows from Theorem \ref{Th1} and the upper bound follows from
(\ref{for2}) by expanding the right in partial fractions.

For the general case, we \ use Proposition \ref{Neron_resol} as in the proof
of Theorem \ref{Th3}. For $m>L$, one gets
\[
N_{m}=\sum_{\overline{x}\in V^{(l-1)}(R_{K})\operatorname{mod}P_{K}^{L}%
}\#\{\overline{y}\in V_{x,L}^{(l-1)}(R_{K})\operatorname{mod}P_{K}^{m-L}\mid
f_{l}(x+\pi^{L}y)\equiv0\operatorname{mod}P_{K}^{m}\}.
\]
If $f_{l}(x+\pi^{L}y)=0$ has no solution in $R_{K}^{n}$, then for $m$ big
enough, the congruence $f_{l}(x+\pi^{L}y)\equiv0\operatorname{mod}P_{K}^{m}$
has no solutions. Thus, there exists a natural number $m_{0}\geq L$ such that
if $m\geq m_{0}$, then $N_{m}$ equals%

\[
\sum_{\overline{x}\in V^{(l)}(R_{K})\operatorname{mod}P_{K}^{L}}%
\#\{\overline{y}\in V_{x,L}^{(l-1)}(R_{K})\operatorname{mod}P_{K}^{m-L}\mid%
\begin{array}
[c]{c}%
f_{l}(x+\pi^{L}y)\equiv\\
0\operatorname{mod}P_{K}^{m}%
\end{array}
\}\text{ }%
\]%
\begin{align}
&  =\sum_{\overline{x}\in V^{(l)}(R_{K})\operatorname{mod}P_{K}^{L}%
}\#\{\overline{y}\in V_{x,L}^{(l-1)}(R_{K})\operatorname{mod}P_{K}^{m-L}\mid%
\begin{array}
[c]{c}%
\pi^{e_{l,x,L}}f_{l,x,L}(y)\equiv\\
0\operatorname{mod}P_{K}^{m}%
\end{array}
\}\nonumber\\
&  =\sum_{\overline{x}\in V^{(l)}(R_{K})\operatorname{mod}P_{K}^{L}%
}\#\{\overline{y}\in V_{x,L}^{(l-1)}(R_{K})\operatorname{mod}P_{K}^{m-L}\mid%
\begin{array}
[c]{c}%
\pi^{e_{l,x,L}-L}f_{l,x,L}(y)\equiv\\
0\operatorname{mod}P_{K}^{m-L}%
\end{array}
\}. \label{lastform2}%
\end{align}

We now prove (1) as follows: since all the submanifolds $V_{x,L}^{(l-1)}%
(R_{K})$ have good reduction $\operatorname{mod}P_{K}$, by the rationality of
$P\left(  t,V_{x,L}^{(l-1)},\pi^{e_{l,x,L}-L}f_{l,x,L}\right)  $, there exists
a constant $M_{0}\left(  x\right)  $ such that all the numbers $N_{m}\left(
V_{x,L}^{(l-1)},\pi^{e_{l,x,L}-L}f_{l,x,L}\right)  $ satisfy a linear
recurrence for $m>M_{0}\left(  x\right)  $ and for all $\overline{x}\in
V^{(l)}(R_{K})\operatorname{mod}P_{K}^{L}$. Therefore, by (\ref{lastform2}),
the numbers $N_{m}$ satisfy a linear recurrence for $m$ big enough, and the
corresponding Poincar\'{e} series is rational.

Finally, we establish the announced bound for the $N_{m}$. Since the bound
holds for the $N_{m}\left(  V_{x,L}^{(l-1)},\pi^{e_{l,x,L}-L}f_{l,x,L}\right)
$, \ using (\ref{lastform2}) and (\ref{obs2})-(\ref{obs3}), one gets%
\begin{equation}
N_{m}\leq C_{0}q^{(n-l+1-\rho)m}m^{m_{\rho}-1}, \label{lastfom1}%
\end{equation}
for all $m>M_{1}>1$, for some positive constants $C_{0}$ and $M_{1}%
\in\mathbb{N}$. We now take%
\[
C:=\max\left\{  \left\{  C_{0}\right\}  \cup\left\{  \frac{N_{m}%
}{q^{(n-l+1-\rho)m}m^{m_{\rho}-1}}\mid1\leq m\leq M_{1}\right\}  \right\}  .
\]
Finally, we can replace $C_{0}$ by $C$ in (\ref{lastfom1}), to obtain a bound
valid for any $m\geq1$.
\end{proof}

\begin{remark}
\label{conjecture}Let $h(x_{1},\ldots,x_{n})$, $g_{i}(x_{1},\ldots,x_{n})$,
$i=1,\ldots,l$ be non-constant polynomials with coefficients in $R_{K}$. Set
\[
S(R_{K})=\left\{  x\in R_{K}^{n}\mid g_{i}(x)=0\text{, }i=1,\ldots,l\right\}
.
\]
Assume that $S(R_{K})$ is a $K$-analytic subset of $R_{K}^{n}$\ of dimension
$d$, see \ \cite{O} for this notion, and that $h\mid_{S(R_{K})}\not \equiv 0$.
We define $N_{m}\left(  S,h\right)  $ as in (\ref{Nm}), and
\[
P\left(  t,S,h\right)  ={\sum\limits_{m=0}^{\infty}}q^{-md}N_{m}\left(
S,h\right)  t^{m}\text{.}%
\]
A general result due to Denef \ implies the rationality of $P\left(
t,S,h\right)  $, see \cite[Theorem 1.6.1]{De}. Indeed, to see this, consider
the set of all $y$ in $R_{K}^{n}$ , such that there exists $x$ in $R_{K}^{n}$
such that%
\[
g_{1}(x)=0,\ldots\text{ },g_{l}(x)=0,\text{ and }ord(x-y)\geq m\text{, and
}ord\left(  h(y)\right)  \geq m.
\]
This set depends on a positive integer $m$, and is definable by a formula in
predicate logic. By Considering its measure and applying Theorem 1.6.1 of
Denef in \cite{De}, the rationality of $P\left(  t,S,h\right)  $ is
established. Note that this result does not give information about the `poles'
of $P\left(  t,S,h\right)  $, and thus estimations for $N_{m}\left(
S,h\right)  $ cannot be obtained from it directly.
\end{remark}

\section{$Z_{\Phi}(\omega,V^{(l-1)},f_{l})$ as a limit of integrals over
$K^{n}$}

As before, we take converging power series $f_{i}\in K[[x_{1},\ldots,x_{n}]]$
for $i=1,\ldots,l-1$ on an open and compact subset $U$ of $K^{n}$, and assume
that $V^{(l-1)}=\{x\in U\mid f_{i}(x)=0$, for $1\leq i\leq l-1\}$ is a closed
$K$-analytic submanifold of $U$. In this section, we specialize $\Theta$ to a
Gel'fand-Leray form $\gamma_{GL}$ on $V^{(l-1)}$. We consider on $K^{n}$ the
measure $|dx|$ associated to the differential form $dx=dx_{1}\wedge
\cdots\wedge dx_{n}$, which is the Haar measure on $K^{n}$ so normalized that
$R_{K}^{n}$ has measure $1$.

The second author proved in \cite[Lemma 2.5]{Z3} that
\[
Z_{\Phi}(\omega,V^{(l-1)},f_{l})=\int_{K^{n}}\Phi\left(  x\right)
\delta\left(  f_{1}\left(  x\right)  ,\ldots,f_{l-1}\left(  x\right)  \right)
\omega\left(  f_{l}\left(  x\right)  \right)  \left\vert dx\right\vert ,
\]
for $\omega\in\Omega_{0}(K^{\times})$, where $\delta$ is the Dirac delta
function. We recall that for a Bruhat-Schwartz function $\theta$,
\[
\int_{K^{n}}\theta(x)\delta(f_{1}(x),\ldots,f_{l-1}(x))|dx|=\lim
_{r\rightarrow+\infty}\int_{K^{n}}\theta(x)\delta_{r}(f_{1}(x),\ldots
,f_{l-1}(x))|dx|,
\]
where the functions $\delta_{r}$ for $r\in\mathbb{N}$ are defined by
\[
\delta_{r}\left(  u\right)  =\left\{
\begin{array}
[c]{lll}%
0 & \text{if} & u\notin\left(  \pi^{r}R_{K}\right)  ^{l-1}\\
&  & \\
q^{r(l-1)} & \text{if} & u\in\left(  \pi^{r}R_{K}\right)  ^{l-1}.
\end{array}
\right.
\]

Note that $\omega\left(  f_{l}\left(  x\right)  \right)  $ is not a
Bruhat-Schwartz function.

From an intuitive point of view, the zeta functions considered here are
defined by concentrating classical Igusa's zeta functions on submanifolds, see
\cite{G-S}, \cite{Z3}.\ In this way, several possible definitions for the
local zeta function along a submanifold appear, and then, a natural question
is to know if these definitions are equivalent. The next proposition answers a
question posed in \cite{Z3}.

\begin{theorem}
\label{pro1}For $\omega\in\Omega_{0}(K^{\times})$,
\[
Z_{\Phi}(\omega,V^{(l-1)},f_{l})=\lim_{r\rightarrow+\infty}\int_{K^{n}}%
\Phi\left(  x\right)  \delta_{r}\left(  f_{1}\left(  x\right)  ,\ldots
,f_{l-1}\left(  x\right)  \right)  \omega\left(  f_{l}\left(  x\right)
\right)  \left\vert dx\right\vert .
\]

\end{theorem}

\begin{proof}
Let $b\in V^{(l-1)}\left(  K\right)  $. After a possibly reordering of the
coordinates $x_{1},\ldots,x_{n}$, there exists a local coordinate change
around $b$ of the form $y=\left(  y_{1},\ldots,y_{n}\right)  =\phi\left(
x\right)  $, with
\[
y_{i}:=\left\{
\begin{array}
[c]{lll}%
f_{i}\left(  x\right)  & \text{if} & i=1,\ldots,l-1\\
&  & \\
x_{i}-b_{i} & \text{if} & i=l,\ldots,n,
\end{array}
\right.
\]
since $V^{(l-1)}\left(  K\right)  $ is a submanifold. For $d$ large enough, we
have an open and compact neighborhood $W$ of $b$ such that $\phi
:W\rightarrow\left(  \pi^{d}R_{K}\right)  ^{n}$ is a $K$-analytic isomorphism
satisfying $\left\vert J(x)\right\vert _{K}=\left\vert J(b)\right\vert _{K}$,
for any $x\in W$, where $J(x)$ is the Jacobian of $\phi$. It is sufficient to
prove the theorem for $\Phi$ equal to the characteristic function of a such
set $W$, since there exists finitely many disjoint subsets $W$ as above on
which $\Phi$ is constant, and which cover $V^{(l-1)}$. From now on, we suppose
that $\Phi$ is the characteristic function of $W$. By using $y=\phi\left(
x\right)  $\ as a change of variables,
\begin{align*}
I_{r}\left(  \omega\right)   &  :=\int_{K^{n}}\Phi\left(  x\right)  \delta
_{r}\left(  f_{1}\left(  x\right)  ,\ldots,f_{l-1}\left(  x\right)  \right)
\omega\left(  f_{l}\left(  x\right)  \right)  \left\vert dx\right\vert \\
&  =\left\vert J(b)\right\vert _{K}^{-1}\int\limits_{\left(  \pi^{d}%
R_{K}\right)  ^{n}}\delta_{r}\left(  y_{1},\ldots,y_{l-1}\right)
\omega\left(  \widetilde{f}_{l}\left(  y_{1},\ldots,y_{n}\right)  \right)
\mid dy_{1}\ldots dy_{n}\mid,
\end{align*}
where $\widetilde{f}_{l}:=f_{l}\circ\phi^{-1}$. Since $\delta_{r}\left(
y_{1},\ldots,y_{l-1}\right)  =q^{r\left(  l-1\right)  }$ if and only if
$y_{i}\in\pi^{r}R_{K}$ for $i=1,\ldots,l-1$, and by assuming that $r\geq d$,
we obtain
\[
I_{r}\left(  \omega\right)  =\left\vert J(b)\right\vert _{K}^{-1}q^{r\left(
l-1\right)  }\int\limits_{\left(  \pi^{r}R_{K}\right)  ^{l-1}\times\left(
\pi^{d}R_{K}\right)  ^{n-l+1}}\omega\left(  \widetilde{f}_{l}\left(
y_{1},\ldots,y_{n}\right)  \right)  \mid dy_{1}\ldots dy_{n}\mid
\]%
\[
=\left\vert J(b)\right\vert _{K}^{-1}\int\limits_{R_{K}^{l-1}\times\left(
\pi^{d}R_{K}\right)  ^{n-l+1}}\omega\left(  \widetilde{f}_{l}\left(  \pi
^{r}y_{1},\ldots,\pi^{r}y_{l-1},y_{l},\ldots,y_{n}\right)  \right)  \mid
dy_{1}\ldots dy_{n}\mid.
\]

Finally, by using that $\omega\in\Omega_{0}\left(  K^{\times}\right)  $ and
the Lebesgue dominated convergence theorem,
\begin{align*}
\lim_{r\rightarrow+\infty}I_{r}\left(  \omega\right)   &  =\left\vert
J(b)\right\vert _{K}^{-1}\int\limits_{\left(  \pi^{d}R_{K}\right)  ^{n-l+1}%
}\omega\left(  \widetilde{f}_{l}\left(  0,\ldots,0,y_{l},\ldots,y_{n}\right)
\right)  \mid dy_{l}\ldots dy_{n}\mid\\
&  =\int\limits_{V^{(l-1)}}\Phi\left(  x\right)  \omega\left(  f_{l}\left(
x\right)  \right)  \mid\gamma_{GL}\mid.
\end{align*}

\end{proof}

\begin{remark}
The previous result is still valid if we replace $\omega(f_{l}(x))$ by any
continuous complex-valued function, in particular,
\begin{align*}
E_{\Phi}(z)  &  =\int\limits_{V^{(l-1)}}\Phi(x)\Psi(zf_{l}(x))|\gamma_{GL}|\\
&  =\int_{K^{n}}\Phi\left(  x\right)  \delta\left(  f_{1}\left(  x\right)
,\ldots,f_{l-1}\left(  x\right)  \right)  \Psi\left(  zf_{l}\left(  x\right)
\right)  \left\vert dx\right\vert \\
&  =\lim_{r\rightarrow+\infty}\int_{K^{n}}\Phi\left(  x\right)  \delta
_{r}\left(  f_{1}\left(  x\right)  ,\ldots,f_{l-1}\left(  x\right)  \right)
\Psi\left(  zf_{l}\left(  x\right)  \right)  \left\vert dx\right\vert .
\end{align*}

\end{remark}

\end{document}